\documentclass[11pt]{article}
\usepackage{graphicx}
\usepackage{latexsym,amssymb}
\usepackage{amsthm}
\usepackage{indentfirst}
\usepackage{amsmath}
\usepackage{color}
\usepackage{xcolor}
\usepackage{hyperref}
\newtheorem{Corollary}[subsection]{Corollary}
\newtheorem*{Theorem A}{Theorem A}

\newtheorem{Definition}[subsection]{Definition}

\newtheorem{Lemma}[subsection]{Lemma}

\newtheorem*{Remark}{Remark}
\newtheorem*{Corollary B}{Corollary B}

\newtheorem*{Claim}{Claim}

\newpage
\begin{document}
\title{\textbf{On the volume of sectional-hyperbolic sets}}
\author {\textbf{Daofei Zhang \& Yuntao Zang}}
\date{}
\maketitle
\begin{abstract}
For a transitive sectional-hypebolic set $\Lambda$ with positive volume on a $d$-dimensional manifold $M$($d\ge3$), we show that $\Lambda=M$ and $\Lambda$ is a uniformly hyperbolic set without singularities.
\end{abstract}


\section{Introduction}

The volume of invariant sets, especially these with some hyperbolicity, play an important role in understanding the dynamics in a measure-theoretical point of view. For diffeomorphisms, Bowen in \cite{Bow75} constructs an example of a hyperbolic horseshoe with positive volume for $C^{1}$ system. But when the differentiability is greater than one, \cite{AlP08} proves the nonexistence of hyperbolic horseshoes with positive volume in the setting of partial hyperbolicity. For flows, people are interested in the measure-theoretical properties of singular-hyperbolic set(a partially hyperbolic set with sectionally expanding central bundle). One important direction involves the existence of SRB measures and other statistical properties, e.g. \cite{APP09}, \cite{LeY17}. For the volume of sectional-hyperbolic set, \cite{AAP07} proves that any proper singualr hyperbolic attractor on three dimensional manifold has zero volume. In this paper, we partially extend this result for the case of higher dimension. More precisely, we prove that a transitive sectional-hyperbolic set with positive volume must be the whole manifold, moreover, it is uniformly hyperbolic and contains no singularities.

%
%
%
%
%
\section{Statement of the result}
Let $M$ be a compact Riemannian manifold with dimension $d\geq3$. The Riemannian structure on $M$ induces a volume measure. Let $\mathfrak{X}^{1}(M)$ be the set of $C^{1}$ vector field on $M$ and let $\mathfrak{X}^{1+}(M)$ be the set of $C^{1}$ vector field whose derivative $DX$ is H$\ddot{o}$lder continuous.

We say $\sigma\in M$ is a singularity if $X(\sigma)=0$. Let $\{X_{t}\}_{t\in\mathbb{R}}$ be the flow induced by $X$. A subset $\Lambda$ is called \emph{invariant} if $X_{t}(\Lambda)=\Lambda$, for any $t\in\mathbb{R}$. A sub-bundle $F$ over an invariant set $\Lambda$ is called \emph{invariant} if $DX_{t}(F)=F$, for any $t\in\mathbb{R}$. An invariant set $\Lambda$ is called \emph{transitive} if the forward orbits of some point $x\in\Lambda$ is dense in $\Lambda$.

\begin{Definition}
	Suppose $\Lambda$ is a compact invariant set for $X\in \mathfrak{X}^{1}(M)$ and $F$ is an invariant sub-bundle over $\Lambda$, we say $F$ is uniformly contracting if there exist $K>0,\ \lambda\in(0,1)$ such that 
	\\$$||DX_{t}| _{F_{x}}||\le K\lambda^{t},\ \forall\ t\ge 0,\ x\in \Lambda.$$\\
	 Similarly, we say $F$ is uniformly expanding if there exist $K>0,\ \lambda\in(0,1)$ such that\\ $$||DX_{-t}|_{F_{x}}||\le K\lambda^{t},\ \forall\ t\ge 0,\ x\in \Lambda.$$\\ 
	Suppose $E$ is another invariant sub-bundle over $\Lambda$. We say $F$ dominates $E$ if there exist $K>0,\ \lambda\in(0,1)$ such that\\
	$$	
	||DX_{t}|_{E_{x}}||\cdot||DX_{-t}|_{F_{X_{t}(x)}}||\le K\lambda^{t},\ \forall\ t\ge 0,\ x\in \Lambda
	.$$\\ 
\end{Definition}
Notice by definition, a uniformly expanding sub-bundle always dominates a uniformly contracting sub-bundle for the positive times.
\begin{Definition}
Suppose $\Lambda$ is a compact invariant set for $X\in \mathfrak{X}^{1}(M)$. We say $\Lambda$ is a partially hyperbolic set if there is a continuous invariant splitting\\
$$ T_{\Lambda}M=E^{s}\oplus E^{cu}$$\\ 
where $E^{s}$ is uniformly contracting and $E^{cu}$ dominates $E^{s}$.
\end{Definition}
In the above setting, the sub-bundle $E^{cu}$ is called \emph{central bundle}.
\begin{Definition}\label{Uniform hyperbolicity}
	Suppose $\Lambda$ is a compact invariant set for $X\in \mathfrak{X}^{1}(M)$. We say $\Lambda$ is a uniformly hyperbolic set if there is a continuous invariant splitting\\
	$$T_{\Lambda}M=E^{s}\oplus\langle X \rangle\oplus E^{u}$$\\
	where $E^{s}$ is uniformly contracting and $E^{u}$ is uniformly expanding.
\end{Definition}
\begin{Definition}
Suppose $\Lambda$ is a compact invariant set for $X\in \mathfrak{X}^{1}(M) $ and $F$ is an invariant sub-bundle over $\Lambda$ with $\dim(F_{x})\geq 2$ for any $x\in\Lambda$. We say $F$ is sectionally expanding if there exist $K>0,\lambda>0$ such that for any $x\in\Lambda$, any non-collinear vectors $v,w$ in $F_{x}$,
\\ $$|\det DX_{t}\mid_{ span\langle v,w \rangle}|\ge Ke^{\lambda t},\ \forall\ t>0.$$\\
\end{Definition}

\begin{Definition}\label{Sectional hyperbolicity}
	A partially hyperbolic set $\Lambda$ is called sectional-hyperbolic if its central bundle $E^{cu}$ is sectionally expanding.
\end{Definition}

\begin{Remark}
By definition, we notice that a uniformly hyperbolic set without singularities is always a sectional-hyperbolic set. Since the uniformly contracting/expanding sub-bundle for $X$ becomes a uniformly expanding/contracting sub-bundle for the reversed vector field $-X$, a uniformly hyperbolic set $\Lambda$ for $X$ is also a uniformly hyperbolic set for $-X$.
\end{Remark}

Our main result is the following:
\begin{Theorem A}\label{Main result}
Assume $X\in\mathfrak{X}^{1+}(M)$ and $\Lambda$ is a transitive sectional-hyperbolic set with positive volume. Then $\Lambda=M$ and $\Lambda$ is a uniformly hyperbolic set without singularities.
\end{Theorem A}
%
%
%
\section{Partial hyperbolicity and sectional expansion}

Let $\Lambda$ be a partially hyperbolic set, for any $x\in\Lambda$ and $\varepsilon>0$, we define the local strong stable manifold \\
$$W_{\varepsilon}^{ss}(x)\triangleq\{y:d(X_{t}(y),X_{t}(x))\le\varepsilon,\  \limsup_{t\to+\infty}\frac{1}{t}\log d(X_{t}(y),X_{t}(x))<0\}.$$\\
The global strong stable manifold is defined by \\
$$W^{ss}(x)\triangleq\bigcup_{t>0}X_{-t}\bigg(W^{ss}_{\varepsilon}(X_{t}(x))\bigg)$$\\
and the stable manifold is defined by \\
$$W^{s}(x)\triangleq\bigcup_{t\in\mathbb{R}}X_{t}\bigg(W^{ss}(x)\bigg).$$\\
If $\Lambda$ is a uniformly hyperbolic set, for $x\in\Lambda$ and $\varepsilon>0$, besides the local strong stable manifold, we can also define its local strong unstable manifold by\\
$$W_{\varepsilon}^{uu}(x)\triangleq\{y:d(X_{-t}(y),X_{-t}(x))\le\varepsilon,\  \limsup_{t\to+\infty}\frac{1}{t}\log d(X_{-t}(y),X_{-t}(x))<0\}.$$\\
The global strong unstable manifold is defined by\\
$$W^{uu}(x)\triangleq\bigcup_{t>0}X_{t}\bigg(W^{uu}_{\varepsilon}(X_{-t}(x))\bigg)$$\\
and the unstable manifold is defined by\\
$$W^{u}(x)\triangleq\bigcup_{t\in\mathbb{R}}X_{t}\bigg(W^{uu}(x)\bigg).$$\\

It is well know that the stable and unstable manifolds are immersed submanifolds as smooth as the system.

The following lemma is from Theorem 2.2 in \cite{AAP07}.
\begin{Lemma}\label{positive volume_Lambda contains SS}
Let $\Lambda$ be a partially hyperbolic set with positive volume for $X\in \mathfrak{X}^{1+}(M)$. Then there exist a point $x\in\Lambda$ and a neighborhood $\gamma$ of $x$ in $W^{ss}(x)$ such that $\gamma\subset\Lambda$. 
\end{Lemma}
The following lemma is from Corollary 2.7 in \cite{BaM11}.
\begin{Lemma}\label{Lorenz-like}
	Let $\Lambda$ be a sectional-hyperbolic set for $X\in \mathfrak{X}^{1}(M)$ and $\sigma$ a singularity in $\Lambda$. Then
	$$W^{ss}(\sigma)\cap\Lambda=\{\sigma\}.$$
\end{Lemma}
\begin{Remark}
	The proof of Lemma \ref{Lorenz-like} in \cite{BaM11} does not use sectional-hyperbolicity, so it holds, in general, even for partially hyperbolic sets.
\end{Remark}
Next we use graph transformation to derive uniform expansion from sectional expansion when there is no singularity.
\begin{Lemma}\label{no_singularity implies hyperbolicity}
	Let $X\in \mathfrak{X}^{1}(M)$ and $\Lambda$ be an invariant compact subset without singularities, assume there is an continuous  sectionally expanding  sub-bundle $E^{cu}\subset T_{\Lambda}M$ over $\Lambda$ with $X(x)\subset E^{cu}$ for any $x\in\Lambda$. Then there is a continuous invariant splitting $E^{cu}=\langle X\rangle\oplus E^{u}$ where $E^{u}$ is uniformly expanding.
\end{Lemma}
\begin{proof}
	We first notice the following two facts:
	\begin{enumerate}
		\item Since $\Lambda$ contains no singularity, there is $C>0$ such that 
		$$C^{-1}\leq \frac{||X(X_{t}(x))||}{||X(x)||}\leq C,\,\forall\,x\in\Lambda,\,t\in\mathbb{R};$$
		\item Since $E^{cu}$ is sectionally expanding, there are $K>0$ and $\lambda<1$ such that $$|\det DX_{-t}(x)|_{span\langle X(x),v\rangle}|\leq K\lambda^{t},\,\forall\,x\in\Lambda,v\in E^{cu}_{x},\,t>0.$$
	\end{enumerate}
	
	Choose any continuous sub-bundle $F$ on $\Lambda$ such that  $E^{cu}=\langle X\rangle \oplus F$(e.g. $F$ is the orthogonal bundles of $\langle X\rangle $ in $E^{cu}$). We use $\pi_{X}$ and $\pi_{F}$ to denote the projection on $\langle X\rangle $ and $F$. 
	\begin{Claim}
		$$\lim_{t\to+\infty}\sup_{x\in\Lambda}||\pi_{F}\circ DX_{-t}(x)|_{F}||=0.$$
	\end{Claim}	
	\begin{proof}[proof of the Claim]

		For the bundle $F$, we have the following properties:
		\begin{itemize}
			\item By continuity, there is some constant $\delta$ such that $$\sin\angle(v_{F},X(x))\geq\delta>0,\,\forall\,x\in\Lambda,v_{F}\in F_{x};$$
			\item By Law of Sines, $$||v||\cdot \sin\angle(v,X(x))=||\pi_{F}(v)||\cdot\sin\angle(\pi_{F}(v),X(x)),\,\forall\,x\in\Lambda,\,v\in E^{cu}.$$
		\end{itemize}
		Hence for any $x\in\Lambda$, any $t>0$ and any $v_{F}\in F_{x}$, we have
		$$\begin{aligned}
		&\quad|\det(DX_{-t}(x)|_{span\langle X(x),v_{F}\rangle})|\\[2 mm]
		&=\frac{||X(X_{-t}(x))||\cdot||DX_{-t}(x)(v_{F})||\cdot\sin\angle(DX_{-t}(x)(v_{F}),X(X_{-t}(x)))}{||X(x)||\cdot||v_{F}||\cdot\sin\angle(v_{F},X(x))}\\
		&= \frac{||X(X_{-t}(x))||\cdot||\pi_{F}(DX_{-t}(x)(v_{F}))||\cdot\sin\angle(\pi_{F}(DX_{-t}(x)(v_{F})),X(X_{-t}(x)))}{||X(x)||\cdot||v_{F}||\cdot\sin\angle(v_{F},X(x))}\\	
		&\leq K\lambda^{t}.
		\end{aligned}$$
		Therefore we have $$\frac{||\pi_{F}(DX_{-t}(x)(v_{F}))||}{||v_{F}||}\leq K\lambda^{t} \frac{||X(x)||\cdot\sin\angle(v_{F},X(x))}{||X(X_{-t}(x))||\cdot\sin\angle(\pi_{F}(DX_{-t}(x)(v_{F})),X(X_{-t}(x)))}\leq CK\delta^{-1}\lambda^{t}.$$
		This completes the proof of the claim.
	\end{proof}
Choose $T>0$ large enough such that
$$||DX_{T}(X_{-T}(x))|_{\langle X\rangle}||\cdot||\pi_{F}\circ DX_{-T}(x)|_{F}||\leq \frac{1}{2},\,\forall\,x\in\Lambda.\eqno(\ast)$$
	
Define $$\widetilde{\mathcal{B}}\triangleq \prod_{x\in\Lambda}\mathcal{L}(F_{x},\langle X(x)\rangle)$$ where $\mathcal{L}(F_{x},\langle X(x)\rangle)$ is the Banach space of all bounded linear maps from $F_{x}$ to $\langle X(x)\rangle$. For any $L\in \widetilde{\mathcal{B}}$, define $||L||\triangleq \sup_{x\in\Lambda} ||L_{x}||$. One can check that under this norm, $\widetilde{\mathcal{B}}$ is a Banach space. We say $L\in \widetilde{\mathcal{B}}$ is a continuous section if for any continuous vector field $\sigma:\Lambda\to F$, the vector field $L\circ\sigma:\Lambda\to\, \langle X\rangle$ is continuous. Let $\mathcal{B}\subset \widetilde{\mathcal{B}}$ be the collection of all continuous sections, one can check that  $\mathcal{B}$ is also a Banach space.

Define a map $\Phi: \mathcal{B}\to \mathcal{B}$ by $$(\Phi(L))_{x}\triangleq DX_{T}\circ L_{X_{-T}(x)}\circ\pi_{F}\circ DX_{-T}(x)-DX_{T}\circ\pi_{X}\circ DX_{-T}(x),\,\forall\,x\in\Lambda.$$ One can check it is well defined and for any $L^{1},L^{2}\in \mathcal{B}$, a direct estimation shows $$||\Phi(L^{1})-\Phi(L^{2})||\leq\sup_{x\in\Lambda}(||DX_{T}(X_{-T}(x))|_{\langle X\rangle}||\cdot||L^{1}-L^{2}||\cdot||\pi_{F}\circ DX_{-T}(x)|_{F}||).$$
	
Then by $(\ast)$, we have $||\Phi(L^{1})-\Phi(L^{2})||\leq \frac{1}{2}||L^{1}-L^{2}||$. By Contraction Mapping Theorem, there is some $L\in\mathcal{B}$ with $\Phi(L)=L$. Let $E^{u}=(F,L(F))$, by the definition of $\Phi$, $E^{u}$ is an $DX_{T}$-invariant continuous bundle on $\Lambda$.  By the sectional expansion and the fact that the angle between $E^{u}$ and $\langle X\rangle$ is uniformly bounded below, we get that $E^{u}$ is uniformly expanding w.r.t. $DX_{T}$. We next show $E^{u}$ is $DX_{t}$-invariant for any $t\in\mathbb{R}$. Going by contradiction, assume there exist $t_{0}$ and $w\in E^{u}_{x}$ such that\\
$$DX_{t_{0}}(w)\notin E^{u}_{X_{t_{0}}(x)}.$$\\
Notice $$||DX_{-nT}(DX_{t_{0}}(w))||=||DX_{t_{0}}(DX_{-nT}(w))||\leq ||DX_{t_{0}}||\cdot ||DX_{-nT}(w)||\rightarrow 0.$$
On the other hand, write $DX_{t_{0}}w=w_{1}+w_{2}\in \langle X\rangle\oplus E^{u}$ with $w_{1}\ne0$. Then $$||DX_{-nT}(DX_{t_{0}}(w))||\geq ||DX_{-nT}(w_{1})||-||DX_{-nT}(w_{2})||,\,\forall\,n\in\mathbb{N}.$$
Since there is no singularity, $||DX_{-nT}(w_{1})||$ is uniformly bounded below. Since $E^{u}$ is uniformly expanding w.r.t. $DX_{T}$, $||DX_{-nT}(w_{2})||\rightarrow 0.$ Hence we get $||DX_{-nT}(DX_{t_{0}}(w))||\nrightarrow 0$, a contradiction.
\end{proof}

By Lemma \ref{no_singularity implies hyperbolicity}, we get the following consequence.
\begin{Corollary} \label{No singularity impies uniform hyperbolicity}
	Suppose $\Lambda$ is a sectional-hyperbolic set for $X\in \mathfrak{X}^{1}(M)$. If $\Lambda$ contains no singularities, then $\Lambda$ is a uniformly hyperbolic set.
\end{Corollary}
\section{Proof of the main result}
\begin{proof}[Proof of the Theorem A]

Since $\Lambda$ has positive volume, by lemma \ref{positive volume_Lambda contains SS}, there exist a point $x\in\Lambda$ and a neighborhood $\gamma$ of $x$ in $W^{ss}(x)$ such that $\gamma\subset\Lambda$. Let $\alpha(x)\subset \Lambda$ be the $\alpha$-limit set of $x$. We first notice for any point $z\in\alpha(x)$, $W^{ss}(z)\subset \Lambda$. Indeed, let $\{t_{k}>0\}$ be an increasing sequence such that $X_{-t_{k}}(x)\to z$, then for any compact part $D\subset W^{ss}(z)$, we can find a sequence of compact subsets $\gamma_{t_{k}}\subset X_{-t_{k}}(\gamma)$ such that $\gamma_{t_{k}}\xrightarrow{C^{0}} D$. This is the consquence of the fact that the backward iteration of strong stable manifold is uniformly expanding. Hence $\alpha(x)$ contains no singularities, for otherwise it would contradict with Lemma \ref{Lorenz-like}.

Let \\
$$\Omega=W^{ss}(\alpha(x))\triangleq\bigcup_{z\in\alpha(x)}W^{ss}(z).$$\\
Notice $\Omega$ is a compact invariant subset of $\Lambda$ and again by Lemma \ref{Lorenz-like}, $\Omega$ contains no singularities. Hence by Corollary \ref{no_singularity implies hyperbolicity}, $\Omega$ is a uniformly hyperbolic set.


\begin{Claim}
	$\Omega=\Lambda$.
\end{Claim}
\begin{proof}
	Since $\Omega$ is invariant, $W^{s}(\Omega)=\Omega$. Let $W^{uu}(\Omega)=\bigcup_{z\in \Omega}W^{uu}(z)$. Since $\Omega$ is a uniformly hyperbolic set, $W^{uu}(\Omega)$ contains an open neighborhood of $\Omega$. Choose $w_{\Lambda}\in\Lambda$ such that $\alpha(w_{\Lambda})=\Lambda$, then there is some $t>0$ such that $X_{-t}(w_{\Lambda})\in W^{uu}(\Omega)$, hence there is some $w_{\Omega}\in \Omega$ such that $\alpha(w_{\Omega})=\alpha(w_{\Lambda})=\Lambda$ which implies $\Omega=\Lambda.$

\end{proof}

By the Claim, we get that $\Lambda$ is a uniformly hyperbolic set without singularities and $W^{s}(\Lambda)=\Lambda$. By the Remark below Definition \ref{Sectional hyperbolicity}, $\Lambda$ is also a sectional-hyperbolic set for $-X$.  Then we can repeat the argument above for $-X$, we also get $W^{u}(\Lambda)=\Lambda$, this implies $\Lambda$ is open and closed, hence $\Lambda=M$. 
\end{proof}
\section{Acknowledgment}

The authors would like to thank their supervisor, Prof. Yang Dawei, for many helpful discussions and encouragements. The authors also thank Prof. Renaud Leplaideur for his useful comments.

\begin{tabular}{l l l}
	\emph{\normalsize Daofei ZHANG}
	& \quad \quad \quad \quad \quad \quad \quad \quad \quad \quad \quad   \quad &
	\emph{\normalsize Yuntao ZANG}
	\medskip\\
	\small School of Mathematical Science
	&& \small Laboratoire de Mathématiques d'Orsay\\
	\small Soochow University
	&& \small CNRS - Universit\'e Paris-Sud\\
	\small Suzhou, 215006, P.R. China
	&& \small Orsay, 91405, France\\
	\small \texttt{dfzhang@stu.suda.edu.cn}
	&& \small School of Mathematical Science\\
	&& \small Soochow University\\
	&& \small Suzhou, 215006, P.R. China\\
	&& \small \texttt{yuntao.zang@math.u-psud.fr}\\
\end{tabular}

\begin{thebibliography}{20}
	\bibitem{AAP07}
	 J. F. Alves, V. Araújo, M. J. Pacifico and V. Pinheiro, On the volume of singular-hyperbolic sets, \textit{Dyn. Syst.} \textbf{22}(2007), no. 3, 249-267. 
	 
	 \bibitem{AlP08}
	 J. F. Alves and F. V. Pinheiro, Topological structure of (partially) hyperbolic sets with positive volume, 
	 \textit{Trans. Amer. Math. Soc.} \textbf{360}(2008), no. 10, 5551-5569. 
	 
	  \bibitem{APP09}
	  V. Araujo, M. J. Pacifico, E. R. Pujals and  M. Viana, Singular-hyperbolic attractors are chaotic(English summary),
	  \textit{Trans. Amer. Math. Soc.} \textbf{361}(2009), no. 5, 2431-2485. 
	  
	 
	 \bibitem{BaM11}
	 S. Bautista and C. A. Morales, Lectures on sectional-Anosov flows, \textit{\url{http://preprint.impa.br/Shadows/SERIE_D/2011/86.html}.}
	 
	 
	 
	
	 
	 \bibitem{Bow75}
	  R. Bowen, A horseshoe with positive measure, \textit{Invent. Math.} \textbf{29}(1975), no. 3, 203-204.
	 
%
	 
	 \bibitem{LeY17}
	  R. Leplaideur and D. Yang, SRB measures for higher dimensional singular partially hyperbolic attractors, \textit{Ann. Inst. Fourier (Grenoble).} \textbf{67}(2017), no. 6, 2703-2717. 
	  
%
%
%
%
%
\end{thebibliography}
\end{document}